\documentclass[12pt]{amsart}
\usepackage{amsmath,amssymb,enumerate,amsfonts,amsthm,graphicx,color}
\usepackage[urlcolor=blue]{hyperref}
\usepackage{verbatim}
\usepackage{comment}
\usepackage{mathtools}
\usepackage{tikz}
\usepackage[margin=1in]{geometry}
\usepackage{graphics}
\usepackage{mathrsfs}
\usepackage{caption}
\usepackage{float}
\usepackage{geometry}
\numberwithin{equation}{section}

\newtheorem{theorem}{Theorem}[section]

\newtheorem{lemma}[theorem]{Lemma}

\newtheorem{corollary}[theorem]{Corollary}
\newtheorem*{theorem*}{Theorem}

\theoremstyle{definition}

\newtheorem{example}[theorem]{Example}

\newtheorem{remark}[theorem]{Remark}

\begin{document}
\title[The Waldschmidt constant of special fat flat subschemes in $\mathbb{P}^N$]{The Waldschmidt constant of special fat flat subschemes in $\mathbb{P}^N$}

\author{Hassan Haghighi, Mohammad Mosakhani}
\address{Hassan Haghighi, Mohammad Mosakhani, \ Faculty of Mathematics, K. N. Toosi
University of Technology, Tehran, Iran.} \email{haghighi@kntu.ac.ir,
mosakhani@aut.ac.ir}



\begin{abstract}
  The purpose of this paper is to construct some special kind of
subschemes in $\mathbb{P}^N$ with $ N\ge 3$, which we call them "fat
flat subschemes" and  compute their Waldschmidt constants. These
subschemes are constructed by adding, in a particular way, a finite
number of linear subspaces of $\mathbb{P}^N$ of many different
dimensions to a star configuration in $\mathbb{P}^N$, with arbitrary
preassigned multiplicities to each one of these linear subspaces, as
well as the star configuration. Among other things, it will be shown
that for every positive integer $d$, there are infinitely many fat
flat subschemes in $\mathbb{P}^N$ with the Waldschmidt constant
equal to $d$. In addition to this, for any two integers $1\le a<b$,
we also construct a fat flat subscheme of the above type in some
projective space $\mathbb{P}^M$, which its Waldschmidt constant is
equal to $b/a$. In addition to these, all non-reduced fat points
subschemes $Z$ in $\mathbb{P}^2$ with the Waldschmidt constants less
than $5/2$ are classified.
\end{abstract}
\maketitle

$\vspace{5mm}$
\section{Introduction}
Let $I$ be a non-trivial homogeneous ideal of
$\mathbb{K}[\mathbb{P}^N]$, the coordinate ring of the projective
space $\mathbb{P}^N$ and let $\alpha(I)=\min \{t \mid I_t\neq 0\}$.
The ideal $I$ can be used to construct two important families of
ideals in $\mathbb{K}[\mathbb{P}^N]$ so that each family, reveals
more algebraic and geometric properties of $I$. One of these
families is constructed by ordinary powers of $I$, which for each
positive integer $m$, the generators of the $m$-th ordinary power of
$I$  can be constructed just by $m$-fold multiplications of elements
of a set of minimal generators of $I$. As an immediate consequence
this construction, one has $\alpha(I^m)=m\alpha(I)$. The other
family is the $m$-th symbolic power of $I$, which is defined as
\begin{align*}
I^{(m)}=\bigcap_{P\in
\text{Ass}(I)}\left(\mathbb{K}[\mathbb{P}^N]\cap
I^m\mathbb{K}[\mathbb{P}^N]_P\right),
\end{align*}
where $\text{Ass}(I)$ is the set of associated prime ideals of $I$
and the intersection is taken in the field of fractions of
$\mathbb{K}[\mathbb{P}^N]$. Since in the process of  constructing of
the $m$-th symbolic power of $I$, one should put away those embedded
primary components of $I^m$ which  emerges during the construction
of the ordinary powers,  the generators of $I^{(m)}$, contrary to
the generators of $I^m$, can not be determined easily, even
computing $\alpha(I^{(m)})$,  is not a simple task.  So, to
investigate the structure of $I^{(m)}$, a natural approach is to
compare it with different ordinary powers of  $I$ which has simpler
structure. At the first step, it is clear that $I^r \subseteq
I^{(m)}$ if and only if $ r\ge m$, however it is not at all clear
that for which pairs of positive integers $(m,r)$ the inclusion
$I^{(m)}\subseteq I^r$ hold. The problem of determining those pairs
of integers $(m,r)$ for which the inclusion $I^{(m)}\subseteq I^r$
hold, is known as {\it the containment problem}, which has
stimulated a great deal of research activities in recent years (see
 \cite{BH, SZSZ} and references therein) to specify the structure of
those pairs $(m,r)$ for which $I^{(m)}\subseteq I^r$. In \cite{ELS},
it is proved that $I^{(hr)} \subseteq I^r$, where
$h=\max\{\text{ht}(P)\mid P\in \text{Ass}(I)\}$. Specially,
$I^{(Nr)}\subseteq I^r$. Hence, for a fixed $r$, there are only a
finite number of integers $m$, such that $I^{(m)} \nsubseteq I^r$.
Moreover, even though the behavior of the sequence
$\{\alpha(I^{(m)})\}_{m\in \mathbb{N}}$ can not be determined
explicitly, but the sequence $\{\frac{1}{m}I^{(m)}\}_{m\in
\mathbb{N}}$ is convergent and its limit
\begin{align*}
\widehat{\alpha}(I)=\lim_{m\longrightarrow \infty}
\frac{\alpha(I^{(m)})}{m}=\inf \big \{\frac{\alpha(I^{(m)})}{m} \mid
m \in \mathbb{N}\big \}.
\end{align*}
is called the Waldschmidt constant of $I$, which is an asymptotic
invariant of $I$. Geometrically, this constant asymptotically
measures the degree of hypersurfaces passing through the closed
subscheme defined by $I$ in $\mathbb{P}^N$. A significant role of
this constant regarding the containment problem is that whenever
$\frac{m}{r}<\frac{\alpha(I)}{\widehat{\alpha}(I)}$, then
$I^{(mt)}\nsubseteq I^{rt}$ for all $t\gg 0$ (see \cite[Lemmas
2.3.1, and 2.3.2]{BH}).

Computing this constant is not an easy task and its value for  for a
general ideal is  not known  at all. The most ideals which this
constant is computed for them are the ones for which the ideal $I$
is the ideal of a finite set of reduced points in projective spaces
(see e.g., \cite{BH, DST1,DST} ), or a finite set of reduced linear
subspaces of projective spaces (see e.g., \cite{BCGHJNSVV, BH,
DHST}). Moreover, whenever $I$ is the defining ideal of a finite set
of points (usually non-reduced), finding lower bounds for
$\alpha(I^{(m)})/m$ are the subject of intriguing conjectures (see
\cite{Chang-Jow, MSSZ}).

In this paper, our contribution to this subject is to compute the
Waldschmidt constant of schemes $X$ defined by the ideals
$I=\bigcap_{i=1}^t I(L_i)^{m_i}$, where $L_i$'s are arbitrary linear
subspaces of $\mathbb{P}^N$ of different dimensions, none of which
contains another, with preassigned multiplicities $m_i\ge 1$. These
type of schemes are referred as {\it fat flat subschemes} (see
\cite[page 401]{BH}). These schemes may be supposed as a natural
generalization of fat points schemes accompanied by what is known as
{\it star configuration}.

Our motivation for investigating these types of schemes was to
generalize the main result of \cite[Teorem B]{MH}, which says that
for every positive integer $d$, there exists a configuration of
points $Z$ in $\mathbb{P}^2$ for which its Waldschmidt constant is
equal to $d$, to higher dimensional projective spaces. Fortunately,
in what follows, this result is extended to a more general family of
closed subschemes of $\mathbb{P}^N$, where $N\ge 3$, i.e., the
subschemes which are supported at the union of finite linear
subspaces of different dimensions in $\mathbb{P}^N$. The main result
of this paper is:
\\\\
{\bf Theorem A.} {\it For every positive integer $d$, there exists a
fat flat subscheme $W$ in $\mathbb{P}^N$, such that for every
positive integer $k$, $\alpha(I(W)^{(k)})=dk$. In particular,
$\widehat{\alpha}(I(W))=d$.}

On the other hand, the question of identifying those closed
subschemes with the given Waldschmidt constant is also of great
interest. In the process of proving the above theorem, we
encountered to a finite set of non-reduced points in $\mathbb{P}^2$
with the Waldschmidt constant equal to two. This observation, at
first, motivated us to identify all non-reduced fat points
subschemes $Z$ in $\mathbb{P}^2$ supported at a finite set of
distinct points, where the defining ideal of them has the least
Waldschmidt constants, i.e., $\widehat{\alpha}(I(Z))=2$.
Fortunately, we were able to identify
 those  closed subschemes of the mentioned type with the Waldschmidt
constant of less than $5/2$. The following theorem describes this
classification.
\\\\
{\bf Theorem B.} {\it Let $Z$ be a finite set of non-reduced closed
points in $\mathbb{P}^2$ with defining ideal $I(Z)$. Then
$\widehat{\alpha}(I(Z))<5/2$ if and only if $Z$ is as one of the
following forms
\begin{enumerate}
\item[a)] $Z=2(p_1+ \dots +p_r) +(p_{r+1} + \dots + p_{r+s})$,
where $r \geq 1$ and $s \geq 0$, such that all of these points are
contained in a line and $\widehat{\alpha}(I(Z))=2$, or;
\item[b)]$Z=(p_1+p_2+ \dots +p_{r+s}) +2p_0$, with $r,s \geq 1$,
such that $p_1, \dots, p_r$ lie on a line $L_1$, and $p_{r+1}, \dots
, p_{r+s}$ lie on another line $L_2$ and $p_0$ is  the intersection
of these two lines and $\widehat{\alpha}(I(Z))=2$, or;
\item[c)]$Z=2p_1+p_2+p_3+p_4$, such that $p_1$ lie on a line and
$p_2,p_3,p_4$ lie on another distinct line and
$\widehat{\alpha}(I(Z))=7/3$.
\end{enumerate}}

At the end of the final section, by giving a few  examples, we show
that to push forward the above classification to $5/2$ and  beyond
of it, more tools and notions are needed to be developed.
\section{A special kind of configuration of linear subspaces in $\mathbb{P}^N$}
Let $N\ge 3$, be an integer. In this section,  we construct a large
family of fat flat subschemes in $\mathbb{P}^N$. These subschemes
 will be used to prove Theorem A.

$(*)$: Let $H_1, \dots , H_s$ be a set of $s$ general hyperplanes in
$\mathbb{P}^N$ and let $e$ be an integer, such that $1 \leq e \leq
N$ and $e \leq s$. For each $i=1, \dots ,{\binom{s}{e}}$, let $L_i=
\bigcap_{j \in \{ i_1 , \dots , i_e \}} H_j$, where $\{ i_1, \dots,
i_e \} \subset \{1, \dots , s\}$. We borrow the notation $S_N(e,s)$
from \cite{BH} to denote the configuration of ${\binom{s}{e}}$
linear subspaces $\{ L_1, \dots , L_{{\binom{s}{e}}} \}$ of
codimension $e$ in $\mathbb{P}^N$, which is known as {\it a star
configuration} (see \cite{GHM} for more properties of these type of
configurations). Let $I=I(S_N(e,s))$ be the defining ideal of
$S_N(e,s)$ and let $m$ be an arbitrary positive integer. We denote
by $mS_N(e,s)$ the closed subscheme of $\mathbb{P}^N$ associated to
the ideal $I^{(m)}$.

As an immediate consequence of \cite[Lemma 2.4.1]{BH},  the
Waldschmidt constant of the ideal $J=I(mS_N(e,s))$ can be computed
as follows.

\begin{lemma}\label{lem}
Let $J$ be the  ideal associated with the scheme $mS_N(e,s)$. Then
$\widehat{\alpha}(J) =\frac{ms}{e}$.
\end{lemma}

\begin{proof}
It is known that $\widehat{\alpha}(J) = \underset{k \rightarrow
\infty}{\lim} \frac{\alpha(J^{(ke)})}{ke}= \underset{k \rightarrow
\infty}{\lim} \frac{\alpha(I^{(m)(ke)})}{ke}= \underset{k
\rightarrow \infty}{\lim} \frac{\alpha(I^{(mke)})}{ke}$. Since $mke$
is a multiple of $e$, by \cite[Lemma 2.4.1]{BH}, we have
$\alpha(I^{(mke)})=\frac{mkes}{e}=mks$. Thus
$\widehat{\alpha}(J)=\frac{ms}{e}$.
\end{proof}

Now, we are ready to construct a large family of fat flat subschemes
in $\mathbb{P}^N$. We also compute the Waldschmidt constant of the
defining ideals of these subschemes.

{\bf Construction: }Let $S_N(e,s)$ be the star configuration in
$\mathbb{P}^N$ constructed as $e$-wise intersection of $s$ general
hyperplanes $\{H_1,\dots, H_s\}$ and let $I$ be its defining ideal.
Let $n_i$, whit $2 \leq i \leq N$, be a non-negative integer. For
each $i=2, \dots, N$, let the set $T_i= \{ M_{i,1}, \dots ,
M_{i,n_i} \}$, consists of $n_i$ distinct linear subspaces of
codimension $i$ in $\mathbb{P}^N$ and contained in
$\bigcup_{j=1}^{s} H_j$. Moreover, let $S_N(e,s)\cap
T_i=\varnothing$ and  $T_i \cap T_j=\varnothing$ for each $2 \leq i
< j \leq N$. Let
\begin{align}\label{equ}
I=\bigcap_{i=2}^{N}\big( \bigcap_{r_i=1}^{n_i}
I(M_{i,r_i})^{(m_{i,r_i})} \big) \bigcap I^{(m)}
=I(W^{\prime})^{(m)} \bigcap I^{(m)},
\end{align}
where $0 \leq m_{i,r_i} \leq [\frac{m}{e}]$ (here, for any real
number $a$, $[a]$ denotes the greatest integer less than or equal to
$a$). Since each linear subspace $M_{i,r_i}$ is a complete
intersection,  every power of the  ideal $I(M_{i,r_i})$ is unmixed
and hence is saturated. So, $I$ is also saturated. Let $W_m$ be the
fat flat subscheme of $\mathbb{P}^N$ corresponds to $I$.
Formally, we may denote the scheme $W_m$ in the following form
\begin{align*}
W_m &:= \big( \sum_{r_2=1}^{n_2} m_{2,r_2} M_{2,r_2} + \dots +
\sum_{r_N=1}^{n_N} m_{N,r_N} M_{N,r_N} \big) + \big( mL_1 + \dots +
mL_{{\binom{s}{e}}} \big)
\\&
=W^{\prime }+mS_N(e,s).
\end{align*}
We also denote the  ideal $I$ defined in \eqref{equ}  by $I(W_m)$.

\begin{remark}\label{W1}
Let $W_m=W^{\prime }+mS_N(e,s)$ be the fat flat subscheme of
$\mathbb{P}^N$ defined as above. This scheme may be non-reduced,
when $m \geq 2$, or may be reduced, whenever $m=1$. In the latter
case, all integers $m_{i,r_i}$ would be zero and hence
$W'=\varnothing$. Consequently, $W_1$ would be just the star
configuration $S_N(e,s)$.
\end{remark}

Let $I(W_m)$ be the corresponding homogeneous ideal of the fat flat
subscheme $W_m=W^{\prime }+mS_N(e,s)$ in $\mathbb{K}
[\mathbb{P}^N]$. In the following theorem, the Waldschmidt constant
of $I(W_m)$ is computed. It would be interesting to note that the
number $\widehat{\alpha}(I(W_m))$ depends only to  the scheme
$mS_N(e,s)$ and it is independent of the scheme $W^\prime$. In
particular, we obtain an infinite family of examples of non-
isomorphic fat flat subschemes  of $\mathbb{P}^N$, such that they
have the same support and  Waldschmidt constant.

\begin{theorem}\label{th}
Let $W_m=W^{\prime }+mS_N(e,s)$ be a fat flat subscheme of
$\mathbb{P}^N$ which constructed as above. Then
$\widehat{\alpha}(I(W_m))=\frac{ms}{e}$.
\end{theorem}

\begin{proof}
We keep the notations as in $(*)$. Let $G$ be an element of
$I(W_m)$. Then $G$ vanishes on $W_m$. Specifically, it vanishes on
$mS_N(e,s)$, hence $G \in I(mS_N(e,s))$, which yields $I(W_m)
\subseteq I(mS_N(e,s))$. In particular, by Lemma \ref{lem},
$\widehat{\alpha}(I(W_m)) \geq \widehat{\alpha}(I(mS_N(e,s)))=
\frac{ms}{e}$. Now, let $F$ be the polynomial $F=H_1^m \dots H_s^m$
in $\mathbb{K}[\mathbb{P}^N]$. Since $F$ vanishes to order $me$ on
each  linear subspace of $\{ L_1, \dots , L_{{\binom{s}{e}}} \}$,
hence $F \in I(S_N(e,s))^{(me)}=I(mS_N(e,s))^{(e)}$. Also, $F$
vanishes to order $m$ on $\bigcup_{i=2}^{N} T_i$. Since we have
assumed that $m_{i,j} \leq [\frac{m}{e}]$, then $F \in
I(W^{\prime})^{(e)}$. Therefore, by \eqref{equ}, we have $F \in
I(W_m)^{(e)}$ and by subadditivity property of the Waldschmidt
constant, $\widehat{\alpha}(I(W_m)) \leq
\frac{\alpha(I(W_m)^{(e)})}{e} \leq \frac{\deg F}{e}=\frac{ms}{e}$.
These imply, $\widehat{\alpha}(I(W_m))= \frac{ms}{e}$.
\end{proof}

In \cite[Theorem 2.4.3]{BH}, it is shown that the Waldschmidt
constant of $S_N(e,s)$ is equal to $s/e$. On the other hand, as the
remark \eqref{W1} shows, the scheme $W_1$ is just $S_N(e,s)$, and
hence by theorem \ref{th} its Waldschmidt constant is $s/e$. This
observation merits to be stated as a corollary of Theorem \ref{th}.
\begin{corollary}\label{gene_BH}
The Waldschmidt constant of the scheme $W_m=W'+mS_N(e,s)$ for $m=1$
is equal to $s/e$.
\end{corollary}

In the following lemma, which  will be used to prove  Theorem A, we
provide a criterion that shows when for a homogeneous ideal $I$ of
$\mathbb{K}[\mathbb{P}^N]$,  the equality $\alpha(I^{(m)})=tm$, for
some integer $t$ and for all $m \geq 1$, holds.  As a result, the
equality $\alpha(I)=\widehat{\alpha}(I)$ would completely determine
the elements of the sequence $\{\alpha(I^{(m)}\}_{m\in \mathbb{N}}$,
as well as we can conclude that when the sequence $\{
\alpha(I^{(m+1)}) - \alpha(I^{(m)}) \}_{m \in \mathbb{N}}$, would be
constant. This sequence, which is known as $\beta-$sequence of $I$,
is used recently to study the geometry of points in $\mathbb{P}^2$
(see \cite{DST, DST1, BCh}).

\begin{lemma}\label{lem 2}
Let $I$ be a homogeneous ideal of $\mathbb{K}[\mathbb{P}^N]$. Then
the equality $\alpha(I^{(m)})= tm$ holds for some integer $t$ and
for all $m \geq 1$ if and only if $\alpha(I)=\widehat{\alpha}(I)=t$.
\end{lemma}
\begin{proof}
Clearly, if $\alpha(I^{(m)})= tm$, for all $m \geq 1$, then
$\alpha(I)= \widehat{\alpha}(I)=t$. Conversely, on the contrary, let
there exist an integer $k$ such that $\alpha(I^{(k)}) \neq tk$. Two
cases may arise. First, assume that $\alpha(I^{(k)}) < tk$. Thus
$t=\widehat{\alpha}(I) \leq \frac{\alpha(I^{(k)})}{k} <
\frac{tk}{k}=t$, which is a contradiction. In the second case, let
$\alpha(I^{(k)}) > tk$. Therefore, $tk < \alpha(I^{(k)}) \leq
\alpha(I^k)=k\alpha(I)=tk$, which gives again a contradiction.
\end{proof}

\begin{example}
To show examples of homogeneous ideals which satisfy the hypothesis
of Lemma \ref{lem 2}, let $W_1=S(e,s)$ be a star configuration. By
\cite[Lemma 2.4.2]{BH}, $\alpha(I(W_1))=s-N+1$ and
$\widehat{\alpha}(I(W_1))=s/e$ by  Corollary \ref{gene_BH}. So, the
equality $\alpha(I(W_1))=\widehat{\alpha}(I(W_1))$ would be hold if
and only if $s=e(N-1)/(e-1)$. Hence, it is enough to choose
parameters $e$ and $N$ such that $e-1$ is a proper divisor of $N-1$.
\end{example}

Now, we are in the position where by using Theorem \ref{th} and
Lemma \ref{lem 2}, we can prove the first main result of this paper.
In the following proof, we keep the notations used in $(*)$.
\\\\
{\bf Proof of Theorem A.} Let $d=st$ be a decomposition of $d$. We
choose an integer $e$ such that $1 \leq e \leq N$ and $e \leq s$.
Let $m=et$ and let
\begin{align*}
W_m &=\big( \sum_{r_2=1}^{n_2} m_{2,r_2} M_{2,r_2} + \dots +
\sum_{r_N=1}^{n_N} m_{N,r_N} M_{N,r_N} \big) + \big( mL_1 + \dots +
mL_{{\binom{s}{e}}} \big)
\\&
=W^{\prime }+mS_N(e,s),
\end{align*}
be the fat flat subscheme of $\mathbb{P}^N$, which is constructed as
in the above construction. With a similar argument  as in the proof
of Theorem \ref{th}, we have $I(W_m) \subseteq I(mS_N(e,s))$. In
particular, by \cite[Lemma 2.4.1]{BH}, $d = \alpha(I(mS_N(e,s)))
\leq \alpha(I(W_m))$. Now, let $G$ be the polynomial $G=H_1^t \cdots
H_s^t$ in $\mathbb{K}[\mathbb{P}^N]$. Then $G$ vanishes to order
$m=et$ on each of the linear subspaces in $\{ L_1, \dots ,
L_{{\binom{s}{e}}} \}$, which implies $G \in I(mS_N(e,s))$.
Moreover, $G$ vanishes to order $t$ on each linear subspace in the
set $T_i$, with $2 \leq i \leq N$. Since we have assumed that
$m_{i,j} \leq t$, it follows that $G \in I(W^\prime)$. Also, since
$W_m=W^{\prime }+mS_N(e,s)$, hence $G \in I(W_m)$, which implies
$\alpha(I(W_m)) \leq \deg G =st=d$. Therefore, $\alpha(I(W_m))=d$.
Also, by Theorem \ref{th}, $\widehat{\alpha}(I(W_m))=d$. Now, Lemma
\ref{lem 2} yields, $\alpha(I^{(m)}(W_m))=dm$ for all $m \geq 1$,
\endproof

\begin{remark}
Most of the fat flat subschemes of $\mathbb{P}^N$ which the
Waldschmidt constant of their defining ideals have been known up to
now  are supported at the linear subspaces of the same dimension and
are reduced. In this paper, the novel point  is that as the above
construction of fat flat subschemes shows, Theorem \ref{th} provides
a large collection of fat flat subschemes of $\mathbb{P}^N$ which
are supported at linear subspaces of different dimensions and the
Waldschmidt constant of each of these fat flat subschemes are equal
to the Waldschmidt constant of  its star configuration component
$mS_N(e,s)$.
\end{remark}

To classify  all configurations of reduced points in $\mathbb{P}^2$
with the Waldschmidt constant less than $9/4$, the authors of
\cite{DST}, introduced a special type of configuration of points, so
called {\it quasi-star configuration} to compute their Waldschmidt
constants. In the following example, we use the proof of Theorem A
to compute the Waldschmidt constant of the defining ideal of a
special kind of fat quasi-star configuration.

\begin{example}
Let $S_2(2,s)=p_1 + \dots + p_{\binom{s}{2}}$ be a star
configuration of points in $\mathbb{P}^2$ which is obtained by
pairwise intersection of $s \geq 2$ general lines $L_1, \dots , L_s$
in $\mathbb{P}^2$. Let $T= \{ q_1, \dots , q_s\}$ be a set of $s$
points, such that for each $i\in \{1,2,\dots,s\}$ the point $q_i$
lie on the line $L_i$ and also these points are not collinear. The
configuration $Z=\sum_{i=1}^{s}q_i + S_2(2,s)$ is called a quasi-
star configuration of points in $\mathbb{P}^2$. Now, let the scheme
$W_2=\sum_{i=1}^{s} q_i + 2S_2(2,s)$ be the fat quasi-star
configuration. Then, by proof of Theorem A,  for every positive
integer $k$, we have $\alpha(I(W_2)^{(k)})=sk$. In particular,
$\widehat{\alpha}(I(W_2)=s$. In other words, for every integer $s\ge
2$, there exists a fat flat subschemes in $\mathbb{P}^2$ with the
Waldschmidt constant equal to $s$.
\end{example}

\begin{center}
\begin{tikzpicture}
\draw [-] [line width=1.pt] (-.8,0)--(5.8,0); \draw [-] [line
width=1.pt] (2.2,2.44)--(.8,-2.74); \draw [-] [line width=1.pt]
(1.5,2.625)--(4.5,-2.925); \draw [-] [line width=1.pt]
(-1,.5)--(4.5,-2.25); \draw [-] [line width=1.pt]
(.5,-2.25)--(6,.4);
\draw [black] (6.2,0) node{$L_1$}; \draw [black] (5.9,.7)
node{$L_2$}; \draw [black] (2.5,2.2) node{$L_3$}; \draw [black]
(1.3,2.25) node{$L_4$}; \draw [black] (-1,.8) node{$L_5$};
\shade[ball color=blue] (0,0) circle (2.8pt); \shade[ball
color=blue] (5.2,0) circle (2.58pt); \shade[ball color=blue] (4,-2)
circle (2.8pt); \shade[ball color=blue] (1,-2) circle (2.8pt);
\shade[ball color=blue] (2,1.7) circle (2.8pt); \shade[ball
color=blue] (1.54,0) circle (2.8pt); \shade[ball color=blue]
(2.918,0) circle (2.8pt); \shade[ball color=blue] (2.5,-1.25) circle
(2.8pt); \shade[ball color=blue] (1.35,-.675) circle (2.8pt);
\shade[ball color=blue] (3.361,-.819) circle (2.8pt);
\shade[ball color=blue] (4.2,0) circle (1.8pt); \draw (4.2,.2)
node{$q_1$};

\shade[ball color=blue] (1.745,.8) circle (1.8pt); \draw (1.45, .8)
node{$q_3$};

\shade[ball color=blue] (2.48,.8) circle (1.8pt);\draw(2.78,.8)
node{$q_4$};

\shade[ball color=blue] (.64,-.31) circle (1.8pt);\draw(.64,-.55)
node{$q_5$};

\shade[ball color=blue] (4.2,-.46) circle (1.8pt); \draw(4.2,-.7)
node{$q_2$};
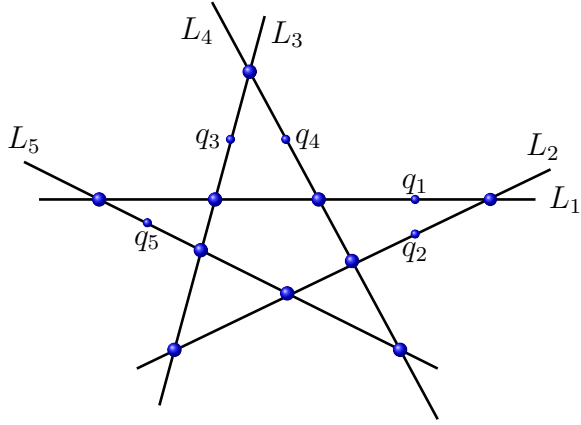
\draw [thick](3,-3.5) node {Figure 1. An illustration of the scheme
$W_2=\sum_{i=1}^{5} q_i + 2S_2(2,5)$ };
\end{tikzpicture}
\end{center}

By now, we were dealing with constructing of fat flat subschemes in
$\mathbb{P}^N$ with the Waldschmidt constant equal to a given
positive integer. It is natural to wonder if this result holds for
positive rational numbers. More specifically,  for a given rational
number $c\ge 1$, is there any fat flat subscheme $Z$ in
$\mathbb{P}^N$ such that $\widehat{\alpha}(I(Z))=c$? The answer is
that, in general, this is not possible. For example, let $c$ be a
rational number such that $2 < c < 9/4$. Main theorem of \cite{DST},
in the reduced case, and Theorem B, in the non-reduced case, show
that there are not any fat points subscheme in $\mathbb{P}^2$ with
the Waldschmidt constant equal to $c$. Therefore, as a simple
consequence of Theorem \ref{th}, we answer this question under the
extra condition that
 the  dimension of the ambient projective space depends on the rational number $c$.

\begin{theorem}\label{cor_wald}
Let $a$ and $b$ be two integers with $1 \leq a < b$. Then there
exists at least a fat flat subscheme $W$ in  some projective space
$\mathbb{P}^N$, such that $\widehat{\alpha}(I(W))=b/a$.
\end{theorem}
\begin{proof}
Choose an integer $N$ such that $a \leq N$. In the first case, let
there exists a non-trivial decomposition of $b$, say $b=sm$, such
that $a \leq s$. Let $W_m=W^{\prime }+mS_N(a,s)$ be the fat flat
subscheme of $\mathbb{P}^N$ constructed as above. Now, by Theorem
\ref{th}, we have $\widehat{\alpha}(I(W))=ms/a=b/a$, which implies
there exist infinitely many configurations of fat flat subschemes in
$\mathbb{P^N}$  such that their Waldschmidt constants are equal to
$b/a$. In the second case, let $b$ has no decomposition as above.
Now, let $W=S_N(a,b)$, i.e., a star configuration in $\mathbb{P}^N$.
Again, by Theorem \ref{th}, we have $\widehat{\alpha}(I(W))=b/a$.
\end{proof}
\begin{remark}
Theorem \ref{cor_wald} and its proof give a method to construct fat
flat subschemes of $\mathbb{P}^N$ whose the Waldschmidt constants
are equal to a given rational number greater than one. We can also,
by choosing suitable ambient spaces, extend this result to obtain
new fat flat subschemes with the same Waldschmidt constant. Let us
describe this method by an example. For the rational number $5/2$,
the proof of Theorem \ref{cor_wald} gives the fat flat subscheme
$W=S_N(2,5)$, for $N \geq 2$, which is a configuration of
codimension two linear subspaces of $\mathbb{P}^N$, with the
Waldschmidt constant equal to $5/2$. We can write $10/4=5/2$. Now
the proof of Theorem \ref{cor_wald} also gives the fat flat
subscheme $W_2=W^{\prime }+ 2S_N(4,5)$ in $\mathbb{P}^N$, for $N
\geq 4$. In this case, the support of $W_2$ is a union of linear
subspaces of codimension 4 in $\mathbb{P}^N$ with the same
Waldschmidt constant, i.e., it is $\widehat{\alpha}(I(W))=5/2$.
\end{remark}

\begin{remark}\label{remark}
Let  $n \geq 2$ be an integer and $W_2=(p_1+\dots + p_{n}) +
2p_0=W^\prime + 2S_2(2,2)$,  be a fat flat subscheme of
$\mathbb{P}^2$. Let the points of support of $W_2$ lie on a
reducible conic $L_1 \cup L_2$, where these points are not collinear
and let the point $p_0$ be the intersection of $L_1$ and $L_2$. In
fact, $W_2$ is a finite set of non-reduced points in $\mathbb{P}^2$.
By Theorem \ref{th}, $\widehat{\alpha}(I(W_2))=2$. In addition to
this, let $V^{'}=m_1q_1 + \dots +m_nq_n$, where $1 \leq m_i \leq 2$
and $q_1 , \dots , q_n$ are distinct points of a line $L$, be
another set of {\it non-reduced} points in $\mathbb{P}^2$. One can
see that $\widehat{\alpha}(I(V^{'}))=2$ (indeed, it is easy to show
that $I(L)^{2m}$, for every positive integer $m$, is the only
homogeneous polynomial of the least degree that belongs to
$I(V^{'})^{(m)}$). These observations motivate us to ask, is there
another set of non-reduced points in $\mathbb{P}^2$, different from
$W_2$ and $V^{'}$, such that its Waldschmidt constant  is equal to
two? In the next section, we show that this question may have a
negative answer.
\end{remark}
\section{ Proof of  Theorem B}
Let $Z=m_1p_1 + \dots + m_np_n$ be a fat points subscheme of
$\mathbb{P}^2$ and let $I(Z)$ be its corresponding homogeneous ideal
in $\mathbb{K}[\mathbb{P}^2]$. Computing the Waldschmidt constant of
an arbitrary homogeneous  ideal in $\mathbb{K}[x_0,x_1,\dots,x_n]$
is not an easy task. In particular, classifying the zero dimensional
subschemes $Z$, based on asymptotic measures such as
$\widehat{\alpha}(I(Z))$, is much more difficult. However, whenever
$Z$ is a reduced scheme, i.e., $m_i=1$ for all $i=1, \dots ,n$, some
results have been obtained in this direction. For example, the Main
Theorem of \cite{DST}, classifies all configurations of reduced
points $Z$ in $\mathbb{P}^2$ for which $\widehat{\alpha}(I(Z)) <
9/4$. In particular, \cite[Theorem A]{MH} identifies those
configurations of reduced points $Z$ in $\mathbb{P}^2$ with
$\widehat{\alpha}(I(Z))=2$, in more details. Moreover, in
\cite[Theorem 2.3]{FGHLMS}, this classification is extended up to
$5/2$.

Our contribution, in this section, is to give a geometric
classification of
 all non-reduced fat points subschemes $Z=m_1p_1 + \dots + m_np_n$,
i.e., those subschemes of $\mathbb{P}^2$ for  which there exists at
least one $i=1, \dots , n$ such that $m_i \geq 2$,
  and $\widehat{\alpha}(I(Z))<5/2$.
To do this, we need to use  some notions from divisor theory on a
smooth projective surface, which is obtained by blowing up  of
$\mathbb{P}^2$ at a  finite set of points. For completeness of the
exposition, we recall some preliminary notions of this theory.

Let $\{ p_1, \dots , p_n \}$ be a finite set of points in
$\mathbb{P}^2$
and let $\pi : X \rightarrow \mathbb{P}^2 $ be 
the blowing up of $\mathbb{P}^2$ at these points. Let
$E_i=\pi^{-1}(p_i)$, with $i=1, \dots,n$, be the exceptional
divisors on $X$ 
and let $L$ denote the pull back of a general line in $\mathbb{P}^2$ in $X$. 
The divisors $L, E_1, \dots, E_n$ give an orthogonal basis for the
divisor class group $\mathrm{CL}(X)$ of $X$ such that
$L^2=-E_i^2=1$. Moreover, recall that a divisor $F$ in
$\mathrm{CL}(X)$ is called nef if $F.D \geq 0$ for every effective
divisor $D$. In addition, if $F=C_1 + \dots +C_r$, where $C_i$'s are
irreducible curves (not necessarily reduced) in $X$, then $F$ is nef
if and only if $F.C_i \geq 0$, for all $i=1, \dots ,r$.

We use the above facts to state the following lemma, which has a key
role in the proof of the main theorem of this section.

\begin{lemma}\label{lemma}
Let $Z=m_1p_1 + \dots + m_np_n$ and $Z^\prime=m^\prime_1p_1 + \dots
+ m^\prime_rp_r$ be two fat points subschemes of distinct points in
$\mathbb{P}^2$, where $r \leq n$ and $1 \leq m_i^\prime \leq m_i$
for all $i=1, \dots, r$. Let $X$ be the blow up of $\mathbb{P}^2$ at
the points $p_1, \dots, p_r$, with orthogonal basis $L, E_1, \dots,
E_r$. Also, let for some integers $t$ and $t_i$, with $i=1, \dots,
r$, the divisor $F=tL-(t_1E_1 + \dots +t_rE_r)$ is nef  on $X$. Then
\begin{center}
$\widehat{\alpha}(I(Z)) \geq \widehat{\alpha}(I(Z^\prime)) \geq
\frac{m_1^\prime t_1 + \dots + m_r^\prime t_r}{t}$.
\end{center}
\end{lemma}

\begin{proof}
Let $G$ be an element of $I(Z)$. Since we have assumed that
$m_i^\prime \leq m_i$, with $i=1, \dots, r$, hence $G$ vanishes on
$Z^\prime$ too. Therefore, $G \in I(Z^\prime)$, which implies $I(Z)
\subseteq I(Z^\prime)$. In particular, $\widehat{\alpha}(I(Z)) \geq
\widehat{\alpha}(I(Z^\prime))$.

To show that  the second inequality holds too,  consider the
effective divisor $D=\alpha(I(Z^\prime)^{(m)})L -m(m^\prime_1E_1 +
\dots + m^\prime_r E_r)$ for some positive integer $m$. Since $F$ is
nef and $D$ is an effective divisor, $F.D \geq 0$. Thus
$\alpha(I(Z^\prime)^{(m)}) \geq m \frac{m_1^\prime t_1 + \dots +
m_r^\prime t_r}{t}$. In particular, $\widehat{\alpha}(I(Z^\prime))
\geq \frac{m_1^\prime t_1 + \dots + m_r^\prime t_r}{t}$.
\end{proof}

By now, we have provided the necessary tools to prove the main
theorem of this
section.
\begin{theorem}{\bf (Theorem B)}.
Let $Z$ be a finite set of non-reduced points in $\mathbb{P}^2$ with
defining ideal $I(Z)$. Then $\widehat{\alpha}(I(Z))<5/2$ if and only
if $Z$
\begin{enumerate}
\item[a)]consists of $r$ double points, plus $s$ simple points,
where $r \geq 1$ and $s \geq 0$, such that all of these points are
contained in a line or;
\item[b)]consists of $r + s$ simple points,
where $r,s \geq 1$, plus one double point, such that $r$ simple
points out of them  lie on a line $L_1$, and the remaining points
lie on another line $L_2$ and the single double point would be at
the intersection of these two lines, or; \item[c)]consists of three
simple points, plus one double point, such that the simple points
out of it lie on a line $L_1$, and the double point lies on another
line $L_2$.
\end{enumerate}
\end{theorem}

\begin{proof}
Let $Z=m_1p_1 + \dots + m_np_n$ be a non-reduced fat points
subscheme of $\mathbb{P}^2$ with $\widehat{\alpha}(I(Z))=2$ and let
$S(Z) =\{ p_1, \dots, p_n \}$ be the support of $Z$. Since we have
assumed that $Z$ is non-reduced,  there exists at least one $i \in
\{1, \dots , n\}$, such that $m_i \geq2$. Now let $Z_i=m_ip_i$. This
immediately implies $\widehat{\alpha}(I(Z_i))=m_i$. By Lemma
\ref{lemma}, $\widehat{\alpha}(I(Z)) \geq
\widehat{\alpha}(I(Z_i))=m_i \geq 2$ and by our assumption,
$\widehat{\alpha}(I(Z))<5/2$. These two inequalities yield $m_i=2$.
Thus, the points of $Z$ are simple points or double points, i.e., $1
\leq m_i \leq 2$, and at least one of them is a double point.

Let $X \longrightarrow \mathbb{P}^2$ be the blowing up of
$\mathbb{P}^2$ at the points $p_1, \dots,p_n$. 
Also, let $L$ be the pullback of a general line in $\mathbb{P}^2$
and let
$E_i=\pi^{-1}(p_i)$. To complete the proof, we proceed to  
several cases.

In the first case, we show that the points of $S(Z)$ lie on a line
or lie on a reducible conic. To do this, it is enough to verify $Z$
can not contain any subscheme consisting of four non-reduced points,
where the points of its support, are in general position.  Suppose,
on the contrary, by renumbering the points if necessary,
$W=m_1p_1+\dots +m_4p_4$ be a non-reduced fat points subscheme of
$Z$, where its support lies on an irreducible conic, say $C$. Since
$W$ is non-reduced  it has at least one double point. Without loss
of generality, we may assume that $m_1=2$. Let
$W^\prime=2p_1+p_2+p_3+p_4$. Now let $\widetilde{C}=
2L-E_1-E_2-E_3-E_4$ be the proper transform of $C$ passing  through
$p_1, \dots, p_4$. Since $\widetilde{C}$ is a prime divisor with
non-negative self-intersection. Hence, it would be nef. Now by Lemma
\ref{lemma}, we have $\widehat{\alpha}(I(Z)) \geq
\widehat{\alpha}(I(W)) \geq \widehat{\alpha} (I(W^\prime)) \geq
\frac{5}{2}$, which contradicts the assumption
$\widehat{\alpha}(I(Z))<5/2$.

By the preceding argument, $Z$ has at least one double point, and
the points of $S(Z)$ lie on a line or lie on a reducible conic. When
these points are collinear, $Z$ is the case $a)$, and we are done.
Thus, for the remainder of the proof, we let the points of $S(Z)$
lie on
a reducible conic. 
In what follows, we show that $Z$ has exactly one double point. On
the contrary, let $Z$ have at least two double points. Without loss
of generality, let $m_1=m_2=2$. Let $Z^\prime$ be the fat points
subscheme $Z^\prime=2p_1+2p_2+p_3$. This situation may be depicted
in the following figure.

\begin{center}
\begin{tikzpicture}
\draw [-] [line width=1.pt] (0,0)--(2.5,0); \draw [-] [line
width=1.pt] (.5,-.5)--(.5,2); \draw [black] (2.8,.1) node{$L_1$};
\draw [black] (.8,2) node{$L_2$}; \shade[ball color=blue] (.5,0)
circle (3pt); \draw [blue] (.9,.4) node{$p_1$}; \shade[ball
color=blue] (1.5,0) circle (3pt); \draw [blue] (1.7,.4) node{$p_2$};
\shade[ball color=blue] (.5,1) circle (1.5pt); \draw [blue] (.9,1)
node{$p_3$}; \draw [thick](1.3,-1.) node {$\mbox{Figure 2} $};
\end{tikzpicture}
\end{center}
Let $G=2L-E_1-E_2-E_3=\widetilde{L}_1+\widetilde{L}_2+E_1$, where
$\widetilde{L}_1=L-E_1-E_2$ is the proper transform of the $L_1$
where passes through $p_1$ and $p_2$ and $\widetilde{L}_2=L-E_1-E_3$
is the proper transform of the $L_2$ where passes through $p_1$ and
$p_3$. Since $G$ is a sum of prime divisors, each of which $G$ meets
non-negatively, thus it is nef and therefore, by Lemma \ref{lemma},
we have $\widehat{\alpha} (I(Z)) \geq \widehat{\alpha} (I(Z^\prime))
\geq \frac{5}{2}$, which contradicts our assumption.

Based on the above arguments, the points of $S(Z)$ lie on a
reducible conic. Thus, at least one of the points of $S(Z)$ lies on
a line, and the other points lie on another distinct line. Moreover,
only one of the points of $Z$ is double. If this double point is at
the intersection of these two lines, $Z$ would be  the case $b)$,
and we are done. Otherwise, $|S(Z)| = n \geq 4$ and two cases may
arise.

 In the first case, let at least two points of $S(Z)$
lie on a line, and the remaining points  lie on another line. In
this case, one can observe that $Z$ contains a (non-reduced) fat
points subscheme supported at four points in general position, which
by  what is argued in the second paragraph of this proof, again
$\widehat{\alpha}(I(Z)) \geq 5/2$, which gives a contradiction. In
the second case, let all points of $S(Z)$, except one of them, which
is a double point of $Z$, lie on a line, i.e., $Z=2p_1+p_2+\dots +
p_n$. This case may be depicted in the following figure.

\begin{center}
\begin{tikzpicture}
\draw [-] [line width=1.pt] (-.5,0)--(4.1,0); \draw [-] [line
width=1.pt] (0,-.5)--(0,1.5); \shade[ball color=blue] (0,.8)
circle(3pt); \draw [blue] (.47,.75) node{$2p_1$};
\shade[ball color=blue] (.8,0) circle (1.5pt); \draw [blue]
(0.8,-.3) node{$p_2$}; \shade[ball color=blue] (1.6,0) circle
(1.5pt); \draw [blue] (1.6,-.3) node{$p_3$};
\draw [blue] (2.4,-.3) node{$\dots$}; \shade[ball color=blue] (3,0)
circle (1.5pt); \draw [blue] (3,-.3) node{$p_n$}; \draw
[thick](2.2,-1.1) node {$\mbox{Figure 3} $};
\end{tikzpicture}
\end{center}
If $n=4$, then $Z$ would be the case $c)$, and we are done. So, we
may assume $n \geq 5$. Now let
\begin{center}
$H=(n-1)L-(n-2)E_1-E_2-\dots -E_n=\widetilde{L}_2+ \dots
+\widetilde{L}_n+E_1$,
\end{center}
where $\widetilde{L}_i=L-E_1-E_i$, with $i=2, \dots ,n$, is the
proper transform of the line passing through $p_1$ and $p_i$. Since
$H$ is the sum of prime divisors, $\widetilde{L}_i$'s and $E_1$, and
each of them meets $H$ non-negatively, thus it is a nef divisor.
Now, by Lemma \ref{lemma}, we have
\begin{equation}\label{88}
\widehat{\alpha}(I(Z)) \geq \frac{3n-5}{n-1}=3-\frac{2}{n-1}.
\end{equation}
Since we have assumed that $n \geq 5$, one can observe that
$\widehat{\alpha}(I(Z)) \geq \frac{5}{2}$, which contradicts our
assumption.

The converse of the theorem for the case of $a)$ and $b)$ follows by
Remark \ref{remark}. For the case of $c)$, let $Z=2p_1+p_2+p_3+p_4$,
and let $L_{1i}$ be the line passing through $p_1$ and $p_i$. Since
$(L_{12}L_{13} L_{14})^2L_1$ is an element of $I(Z)^{(3)}$, so
$\widehat{\alpha}(I(Z)) \leq 7/3$. Also by \eqref{88}, we have
$\widehat{\alpha}(I(Z)) \geq 7/3$.
\end{proof}

\begin{remark}
Let $W^{'}=2p_1+p_2+p_3+p_4$, $Z^{'}=2p_1+2p_2+p_3$ and
$Z=2p_1+p_2+p_3+p_4+p_5$ be three fat points subschemes of
$\mathbb{P}^2$,  which are illustrated in the following figures.

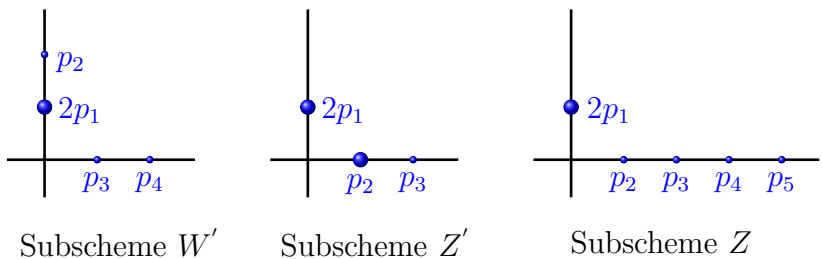
\begin{center}
\begin{tikzpicture}
\draw [-] [line width=1.pt] (-.5,0)--(2,0); \draw [-] [line
width=1.pt] (0,-.5)--(0,2); \shade[ball color=blue] (0,.7)
circle(3pt); \draw [blue] (.47,.65) node{$2p_1$}; \shade[ball
color=blue] (0,1.4) circle(1.5pt); \draw [blue] (.35,1.3)
node{$p_2$}; \shade[ball color=blue] (.7,0) circle(1.5pt); \draw
[blue] (.7,-.3) node{$p_3$}; \shade[ball color=blue] (1.4,0)
circle(1.5pt); \draw [blue] (1.4,-.3) node{$p_4$};
\draw [-] [line width=1.pt] (3,0)--(5.5,0); \draw [-] [line
width=1.pt] (3.5,-.5)--(3.5,2); \shade[ball color=blue] (3.5,.7)
circle(3pt); \draw [blue] (3.97,.65) node{$2p_1$}; \shade[ball
color=blue] (4.2,0) circle(3pt); \draw [blue] (4.2,-.33)
node{$p_2$}; \shade[ball color=blue] (4.9,0) circle(1.5pt); \draw
[blue] (4.9,-.3) node{$p_3$};
\draw [-] [line width=1.pt] (6.5,0)--(10.3,0); \draw [-] [line
width=1.pt] (7,-.5)--(7,2); \shade[ball color=blue] (7,.7)
circle(3pt); \draw [blue] (7.5,.65) node{$2p_1$}; \shade[ball
color=blue] (7.7,0) circle(1.5pt); \draw [blue] (7.7,-.3)
node{$p_2$}; \shade[ball color=blue] (8.4,0) circle(1.5pt); \draw
[blue] (8.4,-.3) node{$p_3$}; \shade[ball color=blue] (9.1,0)
circle(1.5pt); \draw [blue] (9.1,-.3) node{$p_4$}; \shade[ball
color=blue] (9.8,0) circle(1.5pt); \draw [blue] (9.8,-.3)
node{$p_5$};
\draw [thick](1,-1.1) node {$\mbox{Subscheme $W^{'}$} $}; \draw
[thick](4.4,-1.1) node {$\mbox{Subscheme $Z^{'}$} $}; \draw
[thick](8.2,-1.1) node {$\mbox{Subscheme $Z$} $};
\end{tikzpicture}
\end{center}

By the proof of Theorem B, the Waldschmidt constant of each  of
these subschemes is at least $5/2$. It is easy to see that their
Waldschmidt constants are equal to $5/2$. By an argument similar to
the one given in the proof of the converse of the above theorem, it
can be easily seen that the Waldschmidt constant of each one of
these schemes is $\frac{5}{2}$. Also, one can see that there exists
another fat points subscheme $Z$ in $\mathbb{P}^2$ such that
$\widehat{\alpha}(I(Z))=5/2$. For example, let
$W^{''}=W^{'}+q=2p_1+p_2+p_3+p_4+q$, where $q$ is a simple point at
the intersection of two lines. Even though the Waldschmidt constant
of these four subschemes is equal to $5/2$,  trivially, they are not
isomorphic. These observations imply that the classification of all
fat points subschemes $Z$ in $\mathbb{P}^2$ with
$\widehat{\alpha}(I(Z)) \geq 5/2$ might require some new methods to
investigate this problem systematically, and it is the reason 
why we restrict ourselves to the classification of fat points
subschemes with the Waldschmidt constant  less than $5/2$. We hope,
by using new methods, to come back to this issue shortly soon.
\end{remark}

\end{document}